\documentclass[draft]{amsart}

\usepackage{amssymb}

\usepackage{url}

\usepackage{amsrefs}

\newtheorem{theorem}{Theorem}[section]
\newtheorem{lemma}[theorem]{Lemma}
\newtheorem{corollary}[theorem]{Corollary}

\newtheorem{conjecture}[theorem]{Conjecture}

\theoremstyle{definition}

\numberwithin{equation}{section}
\numberwithin{theorem}{section}

\newcommand{\MRnumber}[1]{~{\bf MR~#1}}

\newcommand{\ZZ}{{\mathbb Z}}

\newcommand{\RR}{{\mathbb R}}

\newcommand{\One}{\chi}

\newcommand{\ff}{f\ast f}

\newcommand{\ffi}{\|\ff\|_\infty}
\newcommand{\fficonstant}{1.262}
\newcommand{\fficonstantpertwo}{0.631} 
\newcommand{\sigmaconstant}{1.258883}   

\newcommand{\Ks}{K_{\ss}}

 \DeclareMathOperator{\Spec}{Spec}

\begin{document}






\title[The supremum of autoconvolutions]{The supremum of autoconvolutions, with applications to additive number theory}


%

\author{Greg Martin}
\address{University of British Columbia}
\email{gerg@math.ubc.ca}

\author{Kevin O'Bryant}
\address{City University of New York, College of Staten Island and Graduate Center}
\email{kevin@member.ams.org}


%

\thanks{The first author was supported in part by a grant from the Natural Sciences and Engineering Research Council of Canada. The second author was supported in part by a grant from The City University of New York PSC-CUNY Research Award Program.}

%

\begin{abstract}
We adapt a number-theoretic technique of Yu to prove a purely analytic theorem: if $f\in L^1({\mathbb R}) \cap L^2({\mathbb R})$ is nonnegative and supported on an interval of length $I$, then the supremum of $f\ast f$ is at least $0.631\, \|f\|_1^2/I$. This improves the previous bound of $0.591389\, \|f\|_1^2/I$. Consequently, we improve the known bounds on several related number-theoretic problems. For a set $A\subseteq\{1,2,\dots,n\}$, let $g$ be the maximum multiplicity of any element of the multiset $\{a_1+a_2\colon a_i\in A\}$. Our main corollary is the inequality $g n >  0.631\, |A|^2$, which holds uniformly for all $g$, $n$, and $A$.
\end{abstract}

\maketitle


\section{Introduction}\label{sec:Introduction}

One measure of the ``flatness'' of a nonnegative function $f$ is the ratio of its $L^\infty$ norm to its $L^1$ norm. If $f$ is supported on an interval of length $I$, then we trivially have $\|f\|_\infty \ge \|f\|_1/I$, and equality holds exactly if $f$ is the indicator function of the interval. However, it seems difficult for a convolution of nonnegative functions to be flat throughout its domain, and we seek an improved inequality that reflects this difficulty.

Define the autoconvolution of $f$ to be
    \[f^{\ast 2}(x):=f\ast f(x) := \int_{\RR} f(y) f(x-y)\,dy,\]
and for integers $h>2$ let $f^{\ast h}=f^{\ast(h-1)}\ast f$ be the $h$-fold autoconvolution. Since the $h$-fold autoconvolution is supported on an interval of length $hI$, we have the trivial lower bound $\|f^{\ast h}\|_\infty \ge \|f^{\ast h}\|_1/hI = \|f\|_1^h/hI$. Our main theorem shows that the constant can be improved for even integers $h\ge2$.

\begin{theorem} \label{thm:main}
    If $f\in L^1(\RR) \cap L^2(\RR)$ is nonnegative and supported on an interval of length $I$, and $h$ is an even positive integer, then \[\|f^{\ast h}\|_\infty> \frac{\fficonstant}{hI} \|f\|_1^h.\]
\end{theorem}

By replacing $f(x)$ with $c_1f(c_2x)$, we may assume without loss of generality that $I=1/h$ and $\|f\|_1=1$, in which case $f^{\ast h}$ is supported on an interval of length 1 and the conclusion is simply
    \[\|f^{\ast h}\|_\infty > \fficonstant.\]
In the case $h=2$, this inequality improves the bound $\ffi\ge 1.182778$ of the authors \cite{ExperimentalMathematics}.

Before giving a one-paragraph summary of the proof of Theorem~\ref{thm:main}, we state some of its applications, which will be deduced from the theorem in Section~\ref {sec:Corollaries}. The main topic of \cite{ExperimentalMathematics} is to give upper and lower bounds on the constant appearing in Corollary~\ref{cor:measure}, a result in continuous Ramsey theory concerning {\em centrally symmetric sets}, which are sets $S$ such that $S=c-S$ for some real number $c$. As a consequence of Theorem~\ref{thm:main}, we obtain the following improved lower bound.

\begin{corollary}\label{cor:measure}
Every Lebesgue measurable subset of $[0,1]$ with measure $\epsilon$ contains a centrally symmetric subset with measure $\fficonstantpertwo\epsilon^2$.
\end{corollary}

Although Theorem~\ref{thm:main} is a purely analytic theorem, it has applications to discrete problems, specifically problems in additive number theory. If $A$ is a finite set, let $A+A$ be the multiset $\{a_1+a_2\colon a_i\in A\}$. We call the set $A$ a {\em $B^\ast[g]$ set} if no element of the multiset $A+A$ has multiplicity greater than $g$. When $g$ is even, it is common to call $A$ a $B_2[g/2]$ set. Theorem~\ref{thm:main} yields the following upper bound on the size of $B^\ast[g]$ sets:

\begin{corollary}\label{cor:Bgimprovement}
If $A\subseteq\{1,2,\dots,n\}$ is a $B^\ast[g]$ set, then \(\displaystyle |A| < \sigmaconstant \sqrt{g n}\).
\end{corollary}

In~\cite{CillerueloVinuesa}, by comparison,  Cilleruelo and Vinuesa construct $B^\ast[g]$ sets with $|A|>(\frac{2}{\sqrt{\pi}}-\epsilon) \sqrt{gn}$, provided that $g$ is sufficiently large in terms of $\epsilon>0$, and $n$ is sufficiently large in terms of $g$ and $\epsilon$) . The previous best upper bound is Yu's~\cite{Yu}, who proved that for every fixed even integer $g\ge 2$,
    \[
    \limsup_{n\to\infty} \bigg( \max_{\substack{A\subseteq \{1,2,\dots,n\} \\ A \text{ is a $B^\ast[g]$ set}}} \frac{|A|}{\sqrt{gn}} \bigg) < 1.2649.
    \]
Our corollary improves on Yu's result in three ways. First, we have simplified and shortened the proof considerably. Second, we get a smaller constant. We do use a better kernel function---possibly even the optimal kernel---but that makes very little numerical difference. The bulk of the numerical improvement comes from a lower bound for a diagonal quadratic form in the real parts of Fourier coefficients, inspired by a Fourier coefficient bound used in \cite{2001.Green} and, in a different form, in \cite{ExperimentalMathematics}.
The third improvement is that Corollary~\ref{cor:Bgimprovement} is uniform---it applies to all $g$ and $n$. For example, the uniformity allows us to restate Corollary~\ref{cor:Bgimprovement} as:

\begin{corollary}\label{cor:inverseBg}
If $A\subseteq\{1,2,\dots,n\}$ has cardinality at least $\epsilon n$, then there is an element $s\in A+A$ with multiplicity greater than $\fficonstantpertwo \epsilon^2 n$.
\end{corollary}

These discrete results can also be phrased in terms of {\em Newman polynomials}, which are polynomials all of whose coefficients are 0 or~1. Inequalities for Newman polynomials that relate $p(1)$, $\deg(p)$, and $H(p^2)$ have recently received some attention~\cites{Dubickas,twoauthors}; here $H(q)$ is the {\em height} of the polynomial $q$, defined to be $H(q):=\max\{|q_i| \colon 0\le i \leq \deg(q)\}$ if $q(x)=\sum_{i=0}^{\deg(q)} q_i x^i$. Note that $p(1)$, for a Newman polynomial $p$, counts the number of coefficients of $p$ that equal~1. The main quantity of interest is the ratio
\[R(p):=\frac{H(p^2)(\deg(p)+1)}{p(1)^2},\]
which we can bound from below for all polynomials with nonnegative real coefficients, not just Newman polynomials:

\begin{corollary}\label{cor:polynomials}
If $p$ is a polynomial with nonnegative real coefficients, then $R(p) > \fficonstantpertwo$.
\end{corollary}

Berenhaut and Saidak \cite{twoauthors} have constructed a sequence of polynomials that yields $R(p)\to 8/9\approx 0.89$, and a construction of Dubickas \cite{Dubickas} yields
\mbox{$R(p) \to 5/6\approx 0.83$.}
The authors' work on $B^\ast[g]$ sets \cites{oldarXiv,ExperimentalMathematics} can be rephrased in terms of Newman polynomials; that work shows that if $\epsilon>0$, then there is a sequence of Newman polynomials with $p(1)/\deg(p) \to \epsilon$ and {$R(p)\to \pi /(1+\sqrt{1-\epsilon})^2$.} In particular, there is a sequence of Newman polynomials with $p(1)/\deg(p) \to 0$ and $R(p) \to \pi/4\approx 0.7854$. We conjecture that $\pi/4$ is the best possible constant.

The proof of Theorem~\ref {thm:main} proceeds by forming upper and lower bounds on the integral
    \[
    \int_{\RR} \big(f\ast f(x) +f \circ f(x)\big){K(x)}\,dx,
    \]
where $f\circ f$ is the autocorrelation of $f$ and $K$ is a kernel function to be chosen later. A trivial upper bound is used for $\int_{\RR} f\ast f(x){K(x)}\,dx$, while the upper bound on $\int_{\RR} f\circ f(x){K(x)}\,dx$ uses Parseval's identity to convert the integral into a sum over Fourier coefficients; the central coefficient is pulled out of the sum, and Cauchy-Schwarz is used to bound the remaining terms. The lower bound proceeds by using Parseval's identity, applied to Fourier coefficients with a period smaller than~1, to express the integral in terms of the Fourier coefficients of $f$ and $K$, and then bounding the resulting quadratic form in the real parts of the Fourier coefficients. This strategy reflects Yu's approach from~\cite{Yu} in two ways: noting that the Fourier coefficients of $f\ast f+f\circ f$ have nonnegative real parts, and using Fourier coefficients to two different periods.

After defining some notation and special kernel functions in Section~\ref {sec:Notation}, we prove Theorem~\ref{thm:main} in Section~\ref {sec:Proof}. The corollaries are deduced from the theorem in Section~\ref {sec:Corollaries}; we prove Corollaries~\ref {cor:measure} and~\ref {cor:polynomials} first and then derive the other two corollaries from Corollary~\ref {cor:polynomials}. Finally, in Section~\ref{sec:Improvements} we indicate the extent to which Theorem~\ref{thm:main} could be improved.

\section{Notation and kernel functions}\label{sec:Notation}


For any integrable function $g$, we use the notation
\[
g\ast g(x) := \int_\RR g(y) g(x-y)\, dy
\]
for its autoconvolution and
\[
g\circ g(x) := \int_\RR g(y) g(x+y)\, dy
\]
for its autocorrelation. We define its Fourier transform $\hat g$ by
\[
\hat g(\xi) := \int_\RR g(x) e^{-2\pi i\xi x}\, dx,
\]
and we note that the Fourier transforms of its autoconvolution and autocorrelation satisfy
\[
\widehat{g\ast g}(\xi) = \hat g(\xi)^2
    \quad\text{and}\quad \widehat{g\circ g}(\xi) = \big|\hat g(\xi)\big|^2.
\]
By a {\em probability density function}, or {\em pdf}, we mean a nonnegative function $g$ such that $\|g\|_1=1$. Note that for any pdf $g$, we have $\hat g(0)=1$.

Often we will speak of the $j$th ``Fourier coefficient'' of a function $g$ defined on the real line; by this we mean the $j$th Fourier coefficient of the periodization $\sum_{k\in\ZZ} g(x+k)$ of $g$. The notation $\hat g(j)$ represents two equal quantities: the Fourier transform of $g$ evaluated at the real number $j$, and the $j$th Fourier coefficient of the periodization of~$g$. In any case, the reader can generally replace occurrences of ``Fourier coefficient'' with ``value of the Fourier transform at integers'' without changing the meaning.

As mentioned in the introduction, our method involves using Fourier coefficients corresponding to a period smaller than~1, which we can introduce as a differently normalized Fourier transform. We let $\delta$ be a parameter that lies between 0 and $\frac14$ (we will at the end set $\delta=13/100$), and for notational convenience we set $u:=\delta+\frac12$. We define the differently normalized Fourier transform $\tilde g$ of an integrable function $g$ by
\[
\tilde g(\xi) :=\frac 1u \int_\RR g(x) e^{-2\pi i x \xi/u}\, dx.
\]
Note that in this notation,
\[
\widetilde{g\ast g}(\xi) = u\,\tilde g(\xi)^2
    \quad\text{and}\quad \widetilde{g\circ g}(\xi) = u\,\big|\tilde g(\xi)\big|^2,
\]
while $\tilde g(0)=\frac1u$ for any pdf $g$.

We will invoke Parseval's identity, which is technically a statement involving Fourier coefficients rather than values of Fourier transforms. The following lemma establishes the form of Parseval's identity that will suffice for our purposes.

\begin{lemma}
For $i\in\{1,2\}$, suppose that $g_i$ is a square-integrable function supported on $(-\alpha_i,\alpha_i)$. If $\alpha_1+\alpha_2 \le u$, then
\[
\int_\RR g_1(x) \overline{g_2(x)}\, dx = u \sum_{r\in\ZZ} \tilde g_1(r) \overline{\tilde g_2(r)}.
\]
\label{Parseval u}
\end{lemma}

\begin{proof}
Without loss of generality, we assume that $u\ge \alpha_1 \ge \alpha_2>0$, so that $\alpha_2\le u/2$. Define $G_i(x) = \sum_{k\in\ZZ} g_i(ku+x)$, so that each $G_i$ is a $u$-periodic function and for $r\in\ZZ$
    \begin{align*}
    \tilde g_i(r)
        &=\frac 1u \sum_{k\in\ZZ} \int_{ku-u/2}^{ku+u/2} g_i(x) e^{-2\pi i x r/u}\,dx \\
        &=\frac 1u \sum_{k\in\ZZ} \int_{-u/2}^{u/2} g_i(ku+x) e^{-2\pi i (ku+x) r/u}\, dx \\
        &=\frac 1u \int_{-u/2}^{u/2} \left(\sum_{k\in\ZZ} g_i(ku+x)\right)\,e^{-2\pi i x r/u}\, dx\\
        &=\frac 1u \int_{-u/2}^{u/2} G_i(x)\,e^{-2\pi i x r/u}\, dx.
    \end{align*}
The fact that each $g_i$ is a square-integrable function with compact support implies that each $G_i$ is also square-integrable on any interval of length~$u$. Parseval's identity is thus applicable, giving
\[
\int_{-u/2}^{u/2} G_1(x) \overline{G_2(x)}\, dx = u \sum_{r\in\ZZ} \tilde g_1(r) \overline{\tilde g_2(r)}.
\]
Notice that $G_2(x) = g_2(x)$ throughout the interval $(-\frac u2,\frac u2)$, since $g_2$ is supported on $(-\alpha_2,\alpha_2)$. In particular,
\[
\int_{-u/2}^{u/2} G_1(x) \overline{G_2(x)}\, dx = \int_{-\alpha_2}^{\alpha_2} G_1(x) \overline{g_2(x)}\, dx.
\]
But notice that
\[
G_1(x) = \sum_{k\in\ZZ} g_1(x+ku) = g_1(x)  \quad\text{for }x\in(-\alpha_2,\alpha_2),
\]
since the support of any term $g_1(x+ku)$ with $k\ne0$ is contained in \mbox{$(u-\alpha_1,\infty)$} or $(-\infty,-u+\alpha_1)$, both of which are disjoint from \mbox{$(-\alpha_2,\alpha_2)$} due to the inequality $\alpha_1+\alpha_2\le u$. Therefore
\[
\int_{-\alpha_2}^{\alpha_2} G_1(x) \overline{g_2(x)}\, dx = \int_{-\alpha_2}^{\alpha_2} g_1(x) \overline{g_2(x)}\, dx = \int_\RR g_1(x) \overline{g_2(x)}\, dx,
\]
which establishes the lemma.
\end{proof}

We also use Parseval's identity with respect to the 1-periodic Fourier transform $\hat g$ in the following form:

\begin{lemma}
\begin{sloppypar}
If $g_1$ and $g_2$ are square-integrable functions supported on \mbox{$(-\frac12,\frac12)$}, then
    \[
    \int_\RR g_1(x) \overline{g_2(x)}\, dx = \sum_{r\in\ZZ} \hat g_1(r) \overline{\hat g_2(r)}.
    \]
In particular, $\|g_1\|_2^2 = \sum_{r\in\ZZ} |\hat g_1(r)|^2$.
\end{sloppypar}
\label{Parseval 1}
\end{lemma}

We now turn to describing two kernel functions that will appear in our proofs.
We will need a first kernel function $K$ with the following four properties:
        $K$ is a pdf;
        $K$ is square-integrable;
        $K$ is supported on $(-\delta,\delta)$;
        and $\tilde K(j)$ is real and nonnegative for all integers $j$.
We also need to be able to numerically compute $\|K\|_2^2$ and $\tilde K(j)$ for small $j$ to high accuracy; it will turn out that the primary consideration for $K$ is to have $\|K\|_2^2$ as small as possible.

We define our favored kernel $\Ks$ by setting
\begin{equation}
\Ks(x) := \tfrac1\delta {\ss}\circ {\ss}\big(\tfrac x\delta\big), \quad\text{where }
{\ss}(x) := \begin{cases}
                 \frac{2/\pi}{\sqrt{1-4x^2}}, & \text{if }{-\frac12} < x < \frac 12, \\
                 0, & \text{otherwise.}
               \end{cases}
\label{Ks definition}
\end{equation}
It is obvious that $\Ks$ is nonnegative and supported on $(-\delta,\delta)$, and simple calculus verifies that $\|\Ks\|_1=\|\ss\|_1^2=1$. Since $\Ks$ is defined as an autocorrelation, its Fourier transform $\tilde \Ks(j)$ is automatically real and nonnegative; in particular,
\begin{equation}
\tilde \Ks(j)= u\, | \tilde {\ss} (\delta j)|^2=\frac 1u\, \bigg| J_0\bigg( \frac{\pi \delta j}u \bigg) \bigg|^2,
\label{Ks coefficients}
\end{equation}
where $J_0(t)$ is the Bessel $J$-function of order~$0$. The fact that $|J_0(x)|< 1/\sqrt{x}$ for $x>0$ implies that $\sum_{j\in\ZZ} \tilde \Ks(j)^2$ converges; it then follows from Lemma~\ref{Parseval u} that $\Ks$ is square-integrable. Moreover, we can use the formula~\eqref{Ks coefficients} to accurately compute the $\tilde \Ks(j)$ numerically; we can also accurately compute $\|\Ks\|_2^2=\frac 1\delta \| {\ss} \circ {\ss}\|_2^2< 0.5747/\delta$ using
    \[{\ss} \circ {\ss}(x) = \begin{cases}
                         \frac{2}{\pi^2 |x|} E\big(1-\frac1{x^2}\big), & \text{if }|x|<1, \\
                         0, & \text{otherwise,}
                       \end{cases}\]
where $E(x)$ is the complete elliptic integral of the first kind. We note that in \cite{ExperimentalMathematics} the authors showed that any autocorrelation (or autoconvolution) supported on $(-\delta,\delta)$ has two-norm-squared at least $0.5745/\delta$, and so our particular kernel $\Ks$ is at least nearly optimal.

We will also need a second kernel function $G$ that is $u$-periodic and at least 1 on the interval $[-\frac14,\frac14]$ and that has few non-zero $\sim$-Fourier coefficients, all of which we can compute explicitly. (With respect to this kernel function only, we speak literally of its Fourier coefficients, since it is a legitimate $u$-periodic function.) For this we turn to Selberg's ``magic functions'':

\begin{lemma}
Let $\frac12<u<1$ be a real number and $n>\frac{2u}{2u-1}$ an integer, and define
\[
C_{u,n}(k):= \bigg( 1-\frac kn \bigg) \bigg( \cot \frac{\pi k}n \sin \frac{\pi k}{2u} + \cos \frac{\pi k}{2u} \bigg) + \frac1\pi \sin \frac{\pi k}{2u}
\]
and
\begin{equation}
G_{u,n}(x) := \frac{4u}{2un-2u-n} \sum_{k=1}^{n-1} C_{u,n}(k) \cos \frac{2\pi k x}u.
\label{Gun.def}
\end{equation}
Then $G_{u,n}(x)$ is $u$-periodic, even, and square-integrable on $[-\frac u2,\frac u2]$; moreover, $G_{u,n}(x)\geq 1$ for $-\frac14\leq x \leq \frac14$, and
\[
\tilde{G}_{u,n}(k) = \begin{cases}
         \displaystyle \frac{2uC_{u,n}(|k|)}{2un-2u-n}, & \text{if }1\le |k| < n, \\
         0, & \text{otherwise.}
         \end{cases}
\]
\label{magic.lemma}
\end{lemma}

\noindent Specifically, we will use $G_{63/100,22}$; we note for the record that \[\min_{0\le x \le 1/4} G_{63/100,22}(x) > 1.006.\]

\begin{proof}[Proof of Lemma~\ref{magic.lemma}]
Up to changes of variables that are convenient to our present application, this lemma is directly from Montgomery's account~\cite{10}. Let $K$ be a positive integer. Define $e(u) := e^{2\pi iu}$ and $f(u) := -(1-u)\cot(\pi u) - 1/\pi$, and set \cite[Chapter 1, equations (16) and (18)]{10}
\begin{align*}
\Delta_K(x) &:= \sum_{k=-K}^K \Big( 1-\frac{|k|}K \Big) e(kx) ,\\
V_K(x) &:= \frac1{K+1} \sum_{k=1}^K f\Big( \frac k{K+1} \Big) \sin (2\pi k x).
\end{align*}
Then define \cite[Chapter 1, equation (20)]{10}
\[
B_K(x) := V_K(x) + \frac1{2(K+1)} \Delta_{K+1}(x)
\]
and, for any real numbers $\alpha$ and $\beta$ such that $\alpha\le\beta\le\alpha+1$, define \cite[Chapter 1, equation (21${}^+$)]{10}
\[
S_K^+(x) := \beta - \alpha + B_K(x-\beta) + B_K(\alpha-x).
\]
It follows from Vaaler's Lemma (see \cite[page 6]{10}) that $S_K^+(x) \ge \chi_{[\alpha,\beta]}(x)$; that is, $S_K^+(x) \ge0$ for all $x$ and $S_K^+(x) \ge 1$ for $\alpha\le x\le\beta$. We are interested in the special case $\alpha=-\beta$.

Notice that for any positive integer $n$ and any real number $\beta$, we have
\begin{align*}
V_{n-1}(x-\beta) + V_{n-1}&(-\beta-x) \\
&= \frac1n \sum_{k=1}^{n-1} f\Big( \frac kn \Big) \big( \sin (2\pi k(x-\beta)) + \sin(2\pi k(-\beta-x) \big) \\
&= \frac2n \sum_{k=1}^{n-1} f\Big( \frac kn \Big) \big( {- \sin (2\pi k\beta) \cos (2\pi k x)} \big) \\
&= \frac2n \sum_{k=1}^{n-1} \bigg( \Big(1-\frac kn \Big)\cot \frac{\pi k}n + \frac1\pi \bigg) \sin( 2\pi k\beta) \cos( 2\pi kx).
\end{align*}
Similarly, noting that
\[
\Delta_n(x) = 1 + \sum_{k=1}^n \Big( 1-\frac kn \Big) \big( e(kx) + e(-kx) \big) = 1 + 2 \sum_{k=1}^{n-1} \Big( 1-\frac kn \Big) \cos (2\pi kx),
\]
we have
\begin{align*}
\Delta_n(x-\beta) &+ \Delta_n(-\beta-x) \\
&= 2 + 2\sum_{k=1}^{n-1} \Big( 1-\frac kn \Big) \big( \cos (2\pi k(x-\beta)) + \cos (2\pi k(-\beta-x)) \big) \\
&= 2 + 4\sum_{k=1}^{n-1} \Big( 1-\frac kn \Big) \big( \cos (2\pi k\beta) \cos (2\pi kx) \big).
\end{align*}
Therefore \(B_{n-1} (x-\beta) +B_{n-1}(-\beta-x)\) is equal to
    \[V_{n-1}(x-\beta) + V_{n-1}(-\beta-x) + \frac1{2n} \big( \Delta_n(x-\beta) + \Delta_n(-\beta-x) \big),\]
which we expand into the form
    \[\frac1n + \frac2n \sum_{k=1}^{n-1} \bigg( \Big( 1-\frac kn \Big) \Big( \cot \frac{\pi k}n \sin (2\pi k\beta) + \cos (2\pi k\beta) \Big)
        + \frac{\sin(2\pi k\beta)}\pi \bigg) \cos (2\pi kx).\]
If we now set $\alpha:=-\frac1{4u}$ and $\beta:=\frac1{4u}$ for a real number $\frac12<u<1$, then the function
\begin{multline*}
S_{n-1}^+(x) = \frac1{2u} + \frac1n +\\
    \frac2n \sum_{k=1}^{n-1} \bigg( \Big( 1-\frac kn \Big) \Big( \cot \frac{\pi k}n \sin \frac{\pi k}{2u} + \cos \frac{\pi k}{2u} \Big) + \frac1\pi \sin \frac{\pi k}{2u} \bigg) \cos (2\pi kx)
\end{multline*}
satisfies $S_{n-1}^+(x) \ge 1$ for $|x|\le \frac1{4u}$.

We now stipulate that $n>2u/(2u-1)$, so that $2un-2u-n>0$. Because
\[
G_{u,n}(x) = \frac{2un}{2un-2u-n} \bigg( S_{n-1}^+\Big( \frac xu \Big) - \Big( \frac1{2u}+\frac1n \Big) \bigg),
\]
the inequality $S_{n-1}^+(x) \ge 1$ for $|x|\le \frac1{4u}$ implies that
\[
G_{u,n}(x) \ge \frac{2un}{2un-2u-n} \bigg( 1 - \frac{2u+n}{2un} \bigg) = 1 \qquad\text{for } \bigg|\frac xu\bigg| \le \frac1{4u} ;
\]
that is, $G_{u,n}(x) \ge 1$ for $|x|\le \frac14$ as claimed. The other properties of $G_{u,n}(x)$ follow directly from its definition~\eqref{Gun.def}.
\end{proof}

These kernel functions $\Ks$ and $G_{u,n}$ are quite good, but they have been chosen primarily for their computational convenience rather than for their optimality. For instance, there is no particular reason why an optimal $K$ would {\em have} to be an autocorrelation, and Selberg's functions enjoy many additional properties that are not relevant for our purposes.

\section{Proof of Theorem~\ref{thm:main}}\label{sec:Proof}

We first note that it suffices to prove Theorem~\ref{thm:main} in the case $h=2$: if $h\ge4$ is an even integer, then the function $f^{\ast h/2}$ is supported on an interval of length $hI/2$, and we can apply the theorem for twofold autoconvolutions to $f^{\ast h/2}$, obtaining the required lower bound for the $L^\infty$ norm of \mbox{$(f^{\ast h/2})^{\ast 2} = f^{\ast h}$}. Therefore we may assume that $h=2$ from now on; as noted after the statement of the theorem, we may also assume that $f$ is a pdf supported on an interval of length $\frac12$. In fact, we may assume that $f$ is supported on $({-\frac14},\frac14)$ by replacing $f(x)$ with $f(x-x_0)$ if necessary. For such a function $f$, we need to prove that $\|f\ast f\|_\infty > \fficonstant$.

As described earlier, the proof of Theorem~\ref {thm:main} proceeds by forming upper and lower bounds on the integral
    \[
    \int_{\RR} \big(f\ast f(x) +f \circ f(x)\big){K(x)}\,dx.
    \]
The simple Lemmas~\ref {lem:astupper} and~\ref {lem:circupper} provide the required upper bound, using standard inequalities and 1-periodic Fourier analysis (so the coefficients $\hat f(j)$ appear, for example). The lower bound is provided by the more complicated Lemmas~\ref {lem:equality lemma} and~\ref {lem:improvedlower}; the second kernel function $G$ makes its appearance in Lemma~\ref {lem:improvedlower}, as does $u$-periodic Fourier analysis (including the coefficients $\tilde f(j)$, for example). All four lemmas are stated with general kernels $K$ and $G$ and unspecified parameter $u$, so that other choices can be easily examined for possible improvements to the final outcome. Once these four lemmas are established, the proof of Theorem~\ref{thm:main} can be completed; it is here that we use the specific kernels $\Ks$ and $G_{u,n}$ (with $n=22$) and the specific value $u=63/100$ corresponding to $\delta=13/100$.

\begin{lemma}\label{lem:astupper}
For any pdf $K$, we have
    \( \int_{\RR} (f\ast f(x)){K(x)}\,dx \le \|f\ast f\|_\infty\).
\end{lemma}

\begin{proof}
H\"{o}lder's inequality immediately gives \[\int_{\RR} (f\ast f(x)){K(x)}\,dx \le \|f\ast f\|_\infty \|K\|_1,\] and $\|K\|_1=1$ by assumption.
\end{proof}

\begin{lemma}\label{lem:circupper}
Let $f$ be a square-integrable pdf that is supported on $(-\frac14,\frac14)$, and let $K$ be a square-integrable pdf that is supported on $(-\frac12,\frac12)$. Then
\[\displaystyle \int_{\RR} (f\circ f(x)){K(x)}\,dx  \le 1 +  \sqrt{\ffi-1}\, \sqrt{\|K \|_2^2-1}.\]
\end{lemma}

\begin{proof}
Note that $f\circ f$ is supported on $(-\frac12,\frac12)$; also, $f\circ f$ is square-integrable, since $f$ is a square-integrable function with compact support. Therefore we may apply Lemma~\ref{Parseval 1}:
 \begin{multline*}
\int_\RR (f\circ f(x)){K(x)}\,dx = \int_\RR (f\circ f(x)) \overline{K(x)}\,dx = \sum_{r\in\ZZ} \widehat{f\circ f}(r)\overline{\hat{K}(r)} \\
= \sum_{r\in\ZZ} |\hat f(r)|^2\overline{\hat{K}(r)} = 1 + \sum_{r\not=0} |\hat{f}(r)|^2 \overline{\hat{K}(r)},
 \end{multline*}
since $\hat f(0)=\hat K(0)=1$. The Cauchy-Schwarz inequality now yields
    \begin{align*}
    \int_\RR (f\circ f(x)){K(x)}\,dx
        &\leq 1
            +\bigg( \sum_{r\not=0} |\hat{f}(r)|^4 \bigg)^{1/2} \bigg( \sum_{r\not=0} |\hat{K}(r)|^2 \bigg)^{1/2} \\
        &= 1
            +\bigg( \sum_{r\in\ZZ} |\hat{f}(r)|^4 -1  \bigg)^{1/2}
                \bigg( \sum_ {r\in\ZZ} |\hat{K}(r)|^2 -1\bigg)^{1/2} \\
        &= 1
            +\bigg( \sum_ {r\in\ZZ} |\widehat{f\ast f}(r)|^2 -1 \bigg)^{1/2}
                \bigg( \sum_ {r\in\ZZ} |\hat{K}(r)|^2 -1\bigg)^{1/2}.
    \end{align*}
Since $f\ast f$ is also square-integrable and supported on $(-\frac12,\frac12)$, two applications of Lemma~\ref{Parseval 1} now yield
    \begin{align*}
    \int_\RR (f\circ f(x)){K(x)}\,dx
        &\le 1
            +\left( \|{f\ast f}\|_2^2 - 1  \right)^{1/2}
                \left( \| {K} \|_2^2 - 1 \right)^{1/2} \\
        &\leq 1
            +\left(\ffi-1\right)^{1/2} \left(\|K \|_2^2-1 \right)^{1/2}
    \end{align*}
as claimed.
\end{proof}

Recall that $0<\delta<\frac14$ is a parameter and $u=\delta+\frac12$, and that $\tilde g$ refers to the Fourier transform normalized using the parameter~$u$.

\begin{lemma}\label{lem:equality lemma}
Let $f$ be a square-integrable pdf that is supported on $(-\frac14,\frac14)$, and let $K$ be a pdf that is supported on $(-\delta,\delta)$. Then
    \[
    \int_{\RR} \big(f\ast f(x)+f\circ f(x)\big){K(x)}\,dx = \frac 2u + 2u^2 \sum_{j\not=0} \big(\Re\tilde{f}(j) \big)^2 \Re\tilde{K}(j).
    \]
\end{lemma}

\begin{proof}
The function $f\ast f+f\circ f$ is square-integrable and supported on $(-\frac12,\frac12)$. Since the inequality $\frac12+\delta\le u$ is satisfied, we may apply Lemma~\ref{Parseval u} to obtain
    \begin{align*}
    \int_\RR \big(f \ast f(x)+f\circ f(x)\big){K(x)}\,dx
        &= \int_\RR \big(f \ast f(x)+f\circ f(x)\big)\overline{K(x)}\,dx \\
        &= u \sum_{j\in \ZZ} \big( \widetilde{f\ast f}(j)+ \widetilde{f\circ f}(j) \big) \overline{\tilde{K}(j)} \\
        &= u^2 \sum_{j\in \ZZ} ( \tilde{f}(j)^2+ |\tilde{f}(j)|^2)\tilde{K}(j).
    \end{align*}
As the left-hand side is real, we may take real parts term by term on the right-hand side:
    \begin{align*}
    \int_\RR \big(f \ast f(x)+f\circ f(x)\big){K(x)}\,dx
        &= u^2 \sum_{j\in \ZZ} \Re(\tilde{f}(j)^2+ |\tilde{f}(j)|^2) \Re\tilde{K}(j) \notag \\
        &= 2u^2 \sum_{j\in \ZZ} \big(\Re\tilde{f}(j) \big)^2 \Re\tilde{K}(j),
    \end{align*}
using the fact that $\Re(z^2+|z|^2) = 2(\Re z)^2$ for any complex number~$z$. Continuing,
    \begin{align}
    \int_\RR \big(f \ast f(x)+f\circ f(x)\big)K&(x)\,dx \notag \\
        &= 2u^2 \bigg( \big(\Re\tilde{f}(0) \big)^2 \Re\tilde{K}(0) +  \sum_{j\not=0} \big(\Re\tilde{f}(j) \big)^2 \Re\tilde{K}(j) \bigg) \notag \\
        &= \frac2u + 2u^2 \sum_{j\not=0} \big(\Re\tilde{f}(j) \big)^2 \Re\tilde{K}(j),
        \label {all terms nonnegative}
    \end{align}
since $\tilde f(0) = \tilde K(0) = \frac1u$. This establishes the lemma.
\end{proof}

We comment that Lemma~\ref {lem:equality lemma} implies
    \[
    \int_{\RR} \big(f\ast f(x)+f\circ f(x)\big){\Ks(x)}\,dx \ge \frac 2u,
    \]
since $\tilde \Ks(j)\ge 0$ for all $j$. Combining this lower bound with the upper bounds of Lemmas~\ref {lem:astupper} and~\ref {lem:circupper} applied with $K=\Ks$ gives
    \begin{multline*}
    \frac 2u \leq \int_{\RR} \big(f\ast f(x)+f\circ f(x)\big){\Ks(x)}\,dx \\
    \le \ffi + 1 + \sqrt{\ffi -1}\sqrt{0.5747/\delta-1}\,;
    \end{multline*}
setting $\delta := 0.1184$ and solving for $\ffi$ yields $\ffi \geq 1.25087$, already an improvement over earlier work.

We should be able to do better, of course, than simply throwing away all the $j\ne0$ terms on the right-hand side of Lemma~\ref {lem:equality lemma}. To do so and thus achieve the statement $\ffi \geq \fficonstant$ of Theorem~\ref{thm:main}, we utilize our second kernel function $G$ and the following lemma. Let $\Spec(G):=\{j\in\ZZ\colon \tilde{G}(j)\not=0\}$ denote the support of the $\sim$-Fourier series for the function~$G$.

\begin{lemma}\label{lem:improvedlower}
Let $f$ be a square-integrable pdf supported on $(-\frac14,\frac14)$. Let $G$ be an even, real-valued, square-integrable function that is $u$-periodic, takes positive values on $(-\frac14,\frac14)$, and satisfies $\tilde G(0)=0$. Let $K$ be a function supported on $(-\delta,\delta)$, with $\tilde K(j)\geq 0$ for all integers~$j$ and $\tilde K(j)>0$ for $j\in\Spec(G)$. Then
    \[u^2 \sum_{j\not=0} \big(\Re\tilde{f}(j) \big)^2 \Re\tilde{K}(j)
            \geq \left( \min_{0\leq x \leq 1/4} G(x) \right)^2\,\cdot\,\left( \sum_{j\in\Spec(G)} \frac{\tilde G(j)^2}{\tilde K(j)}\right)^{-1}.\]
\end{lemma}

\begin{proof}
We observe that
    \begin{multline*}
        \min_{0\leq x \leq 1/4} G(x)
        = \left( \min_{-1/4 \leq x \leq 1/4} G(x) \right) \int_{-1/4}^{1/4} f(x)\,dx \\
        \le \int_{-1/4}^{1/4} f(x) G(x)\,dx
        = \int_\RR f(x) \overline{G(x)}\,dx,
    \end{multline*}
since $f$ is supported on $(-\frac14,\frac14)$. Lemma~\ref{Parseval u} then gives
    \begin{equation*}
        \min_{0\leq x \leq 1/4} G(x)
        \le u \sum_{j=-\infty}^\infty \tilde f(j) \overline{\tilde G(j)}.
    \end{equation*}
Taking real parts of both sides, and noting that $\tilde G(j)$ is real under the hypotheses on $G$, yields
    \begin{equation*}
        \min_{0\leq x \leq 1/4} G(x)
        \le u \sum_{j=-\infty}^\infty \Re \tilde f(j) \cdot \tilde G(j)
        = u\sum_{j\in\Spec(G)} \Re \tilde f(j) \cdot \tilde G(j).
    \end{equation*}
We now have, using Cauchy-Schwarz in the middle inequality,
    \begin{align*}
    \left( \min_{0\leq x \leq 1/4} G(x) \right)^2
        &\le u^2 \bigg( \sum_{j\in\Spec(G)} \Re \tilde f(j) \cdot \tilde G(j) \bigg)^2\\
        &= u^2 \bigg( \sum_{j\in\Spec(G)} \Re \tilde f(j) \sqrt{\tilde K(j)} \cdot \frac{\tilde G(j)}{\sqrt{\tilde K(j)}} \bigg)^2 \\
        &\leq u^2 \bigg( \sum_{j\in\Spec(G)} \big(\Re \tilde f(j)\big)^2 \tilde K(j) \bigg)\,
                \bigg( \sum_{j\in \Spec(G)} \frac{\tilde G(j)^2}{\tilde K(j)} \bigg)\\
        &\leq \bigg(u^2 \sum_{j\not=0} \big(\Re \tilde f(j)\big)^2 \tilde K(j) \bigg)\,
                \bigg( \sum_{j\in \Spec(G)} \frac{\tilde G(j)^2}{\tilde K(j)} \bigg),
    \end{align*}
where we have used the hypothesis that $\tilde K(j)\geq 0$ for all integers~$j$ and $\tilde K(j)>0$ for $j\in\Spec(G)$ in the last inequalities (as well as $\tilde G(0)=0$, so that $0\notin\Spec(G)$).
\end{proof}

\begin{proof}[Proof of Theorem~\ref{thm:main}]
We assume, as we have observed we may, that $h=2$ and that $f$ is a pdf supported on $({-\frac14},\frac14)$. We would like to apply Lemmas~\ref{lem:astupper}--\ref{lem:improvedlower} with the choices $\delta:=13/100$, $u:=\delta+1/2=63/100$, $K:=\Ks$ as defined in equation~\eqref{Ks definition}, and $G:=G_{u,22}$ as defined in equation~\eqref{Gun.def}. We already saw in Section~\ref{sec:Notation} that $\Ks$ is a square-integrable pdf supported on $(-\delta,\delta)$ whose $\sim$-Fourier coefficients are nonnegative; a calculation shows that $\tilde K(j)>0$ for $1\le|j|\le21$. We also know from Lemma~\ref{magic.lemma} that $G_{u,22}$ is an even, real-valued, square-integrable $u$-periodic function that takes positive values on $(-\frac14,\frac14)$; moreover, $\tilde G_{u,22}(0)=0$ and $\Spec(G_{u,22}) = \{j: 1\le|j|\le21 \}$. Therefore all the hypotheses of Lemmas~\ref{lem:astupper}--\ref{lem:improvedlower} are satisfied.

By Lemmas~\ref{lem:astupper} and~\ref{lem:circupper}, we have
\begin{align}
    \int_{\RR} \big(f\ast f(x) +f \circ f(x)\big)\Ks&(x)\,dx \notag \\
    &\leq \ffi + 1 + \sqrt{\ffi -1}\sqrt{\|\Ks\|_2^2-1} \notag \\
    &\le \ffi + 1 + 2L\sqrt{\ffi -1}, \label{2L}
\end{align}
where
    \[
    L:= 0.9248 > \frac12 \sqrt{0.5747/\delta-1} > \frac12 \sqrt {\|\Ks\|_2^2-1}.
    \]
(The constant in equation~\eqref{2L} has been called $2L$ rather than $L$ simply for later convenience.) On the other hand, by Lemmas~\ref{lem:equality lemma} and~\ref{lem:improvedlower}, we have
    \begin{multline*}
    \int_{\RR} \big(f\ast f(x) +f \circ f(x)\big){\Ks(x)}\,dx \\
    {}\geq \frac 2u +2\bigg( \min_{0\leq x \leq 1/4} G_{u,22}(x) \bigg)^2\,\cdot\,\bigg( \sum_{j\in\Spec(G_{u,22})} \frac{\tilde G_{u,22}(j)^2}{\tilde \Ks(j)}\bigg)^{-1} \ge R,
    \end{multline*}
where
    \begin{align*}
    R:=3.20874 &< \frac{2}{u}+ 2(1.006)^2 \bigg( 2 \sum_{j=1}^{21} \frac{(2uC_{u,22}(j)/(44u-2u-22))^2}{J_0(\pi \delta j/u)^2/u}\bigg)^{-1} \\
    &= \frac{2}{u}+ \frac{(1.006)^2 (22u-u-11)^2}{u^3} \bigg( \sum_{j=1}^{21} \frac{C_{u,22}(j)^2}{J_0(\pi \delta j/u)^2}\bigg)^{-1}.
    \end{align*}
Combining these inequalities yields
    \begin{equation}\label{eq:inequality}
    \ffi+1+2L\sqrt{\ffi-1} \ge R.
    \end{equation}

Using the notation $Q=\sqrt{\ffi-1}$, equation~\eqref{eq:inequality} becomes $(Q^2+1)+1+2LQ \ge R$, or equivalently $Q^2 + 2LQ + (2-R)\ge0$. The discriminant $4L^2-4(2-R)$ is nonnegative, and so $Q$ must be at least as large as the positive root of $x^2+2Lx+(2-R)$, which is $\big({-2L}+\sqrt{4L^2-4(2-R)}\big)/2$. This establishes that
    \[\sqrt{\ffi-1} \ge \sqrt{L^2+R-2}-L,\]
and it follows from the fact that $\sqrt{L^2+R-2}-L$ is nonnegative that
    \[\ffi \geq \left(\sqrt{L^2+R-2}-L\right)^2+1 > \fficonstant,\]
as claimed.
\end{proof}

\section{Deriving the corollaries}\label{sec:Corollaries}

We now derive Corollaries~\ref {cor:measure} and~\ref {cor:polynomials} from Theorem~\ref{thm:main}, then derive Corollaries~\ref {cor:Bgimprovement} and~\ref{cor:inverseBg} from Corollary~\ref {cor:polynomials}. We let
\[\One_S(x) :=
\begin{cases}1, &\text{if }x\in S,\\0, &\text{if }x\notin S
\end{cases}\]
denote the indicator function of a set $S$, and we let $\mu$ denote Lebesgue measure on~$\RR$.

\begin{proof}[Proof of Corollary~\ref{cor:measure}.]
Let $B$ be a set with measure $\epsilon$ supported in $[0,1]$, so that $\|\One_B\|_1 = \epsilon$.  Theorem~\ref{thm:main} applied with $h=2$ tells us that
    \[\|\One_B \ast \One_B\|_\infty > \frac{\fficonstant}{2\cdot 1}\,\|\One_B\|_1^2 = \fficonstantpertwo\epsilon^2.\]
On the other hand,
    \[\One_B \ast \One_B(x) = \int_{\RR} \One_B(y) \One_B(x-y)\,dy = \mu(C_x) ,\]
where $C_x = B \cap (x-B)$ is the largest centrally symmetric subset of $B$ with center $x/2$. Therefore $\|\One_B \ast \One_B\|_\infty$ is the measure of the largest centrally symmetric subset of $B$, which establishes the corollary.
\end{proof}

\begin{proof}[Proof of Corollary~\ref{cor:polynomials}.]
Suppose that $p(x)=\sum_{i=0}^{\deg(p)} p_i x^i$ with all of the $p_i$ nonnegative, and set
    \[g(x):=\sum_{i=0}^{\deg(p)} p_i \One_{(i-1/2,i+1/2)}(x).\]
Clearly $g$ is nonnegative and supported on the interval $({-\frac12},{\deg(p)}+\frac12)$, and $\|g\|_1 = \sum_{i=0}^{\deg(p)} p_i = p(1)$. Theorem~\ref{thm:main} applied with $h=2$ gives
    \begin{equation}\label{eqn:Newmancor}
      \|g\ast g\|_\infty > \frac{\fficonstant}{2I} \|g\|_1^2 = \frac{\fficonstantpertwo}{\deg(p)+1}p(1)^2.
    \end{equation}
However, note that
    \begin{align*}
    g\ast g(x) &= \sum_{i=0}^{\deg(p)} \sum_{j=0}^{\deg(p)} p_ip_j \One_{(i-1/2,i+1/2)}\ast\One_{(j-1/2,j+1/2)}(x) \\
                    &= \sum_{i=0}^{\deg(p)} \sum_{j=0}^{\deg(p)} p_ip_j \Lambda(x-i-j),
    \end{align*}
where $\Lambda(x):=\max\{0,1-|x|\}$ is a tent function with corners only at integers. In particular, $g\ast g$ is a piecewise linear function with corners only at integers, and so $\|g\ast g\|_\infty = \max\{g\ast g(k)\colon k\in\ZZ\}$. Moreover, for $k$ an integer
    \[g\ast g(k) = \sum_{i=0}^{\deg(p)} \sum_{j=0}^{\deg(p)} p_ip_j \Lambda(k-i-j) = \mathop{\sum_{i=0}^{\deg(p)} \sum_{j=0}^{\deg(p)}}_{i+j=k} p_i p_j\]
(since $\Lambda(0)=1$ while $\Lambda(n)=0$ for every nonzero integer $n$), which is exactly the coefficient of $x^k$ in $p(x)^2$. Therefore $\|g \ast g\|_\infty = H(p^2)$, and so the inequality~\eqref {eqn:Newmancor} is equivalent to the assertion of the corollary.
\end{proof}

\begin{proof}[Proof of Corollaries~\ref{cor:Bgimprovement} and~\ref{cor:inverseBg}.]
Fix $A\subseteq\{0,1,\dots,n-1\}$, and let $g$ be the maximum multiplicity of an element of the multiset $A+A$, so that $A$ is a $B^\ast[g]$ set. Define $p(x)=\sum_{a\in A} x^a$, and observe that $p(1)=|A|$. Note that the coefficient of $x^k$ in $p(x)^2$ is exactly the multiplicity of $k$ as an element of $A+A$, so that $H(p^2)=g$. Corollary~\ref{cor:polynomials} then yields
\[
\fficonstantpertwo < R(p) = \frac{H(p^2)(\deg(p)+1)}{p(1)^2} \le \frac{gn}{|A|^2},
\]
since $\deg(p)\le n-1$. Solving this inequality for $|A|$ establishes Corollary~\ref{cor:Bgimprovement}; on the other hand, setting $|A|=\epsilon n$ and solving the inequality for $g$ establishes Corollary~\ref{cor:inverseBg}.
\end{proof}

\section{Open problems}\label{sec:Improvements}

We conclude by mentioning a few open problems associated to twofold autoconvolutions. The function
\[
h(x)=\begin{cases}
1/\sqrt{2x}, &\text{if }0<x<1/2, \\
0, &\text{otherwise}
\end{cases}
\]
is a pdf whose autoconvolution is
\[
h\ast h(x)=\begin{cases}
\pi/2, &\text{if }0<x\le1/2, \\
\pi/2-2\arctan\sqrt{2x-1}, &\text{if }1/2<x<1, \\
0, &\text{otherwise.}
\end{cases}
\]
In particular, $\|h\ast h\|_1=\|h\|_1^2=1$, while $\|h\ast h\|_\infty = \pi/2$ and $\|h\ast h\|_2^2 = \log4$ (here $\log$ is the natural logarithm). We believe that this function is extremal in two ways. First, this function demonstrates that the constant $\fficonstant$ in Theorem~\ref{thm:main} cannot be increased beyond $\pi/2\approx1.5708$, and we believe that this constant represents the truth:

\begin{conjecture}\label{conj:piover2}
If $g\in L^1(\RR) \cap L^2(\RR)$ is nonnegative and supported on an interval of length $I$, then
\[
\|g\ast g\|_\infty \geq \frac{\pi/2}{2I}\|g\|_1^2.
\]
\end{conjecture}

Dubickas~\cite{Dubickas} gives a sequence of functions supported on $[0,1]$, which take only the values $\pm1$ on that interval, satisfying $\|f\ast f\|_\infty \to 0$. That example shows that the hypothesis of nonnegativity in Conjecture~\ref {conj:piover2} is necessary.

Second, H\"older's inequality gives the upper bound $\|g\|_2^2 \leq \|g\|_\infty \|g\|_1$, which can be an equality if $g$ is a multiple of an indicator function. However, when we apply this inequality in the last step of the proof of Lemma~\ref{lem:circupper}, we apply it to the function $f\ast f$, which seems far from being constant on its support. We conjecture that an improvement is possible here, with the function $h$ again providing the best possible constant. This conjecture would imply that the constant $\fficonstant$ in Theorem~\ref{thm:main} could be improved to $1.3674$.

\begin{conjecture}\label{conj:log16overpi}
If $g\in L^1(\RR) \cap L^2(\RR)$ is nonnegative, then \[ \|g\ast g\|_2^2 \leq \frac{\log 16}{\pi}\,\|g\ast g\|_\infty \,\|g\ast g\|_1.\]
\end{conjecture}

The analogous discrete inequality cannot be improved in this manner. Set $\sum_i q_i x^i=\left(\sum_i^{\deg(p)} p_i x^i\right)^2$, where $p_i\ge 0$; the inequality that we refer to is \(\sum q_i^2 \leq \max q_i \,\cdot\, \sum q_i\).
To see that this cannot be improved, set $p_i=1$ if $i=2^k$  with $0\leq k < N$, and set $p_i=0$ otherwise. In this case, $\sum q_i^2 = 2N^2-N$, while $\max q_i =2$ and $\sum q_i = N^2$, and the inequality is seen to be asymptotically sharp.

Finally, consider a pdf $g$ that is supported on an interval of length~1. Theorem~\ref{thm:main} indicates that $\|g^{\ast h}\|_\infty > \fficonstant / h$ for even integers $h$, while the central limit theorem implies that much more is true: for large integers $h$ there is a constant $c>0$ such that $\|g^{\ast h}\|_\infty \geq c/\sqrt{h}$.
It seems likely that $1/\sqrt h$ is the true rate of decay as $h\to\infty$, at least under some assumption on $g$ such as piecewise continuity, but we have not succeeded in proving this.

\end{document}